\documentclass[final,nomarks]{dmtcs-episciences}
\usepackage[utf8]{inputenc}
\usepackage{amsmath}
\usepackage{amssymb}
\usepackage{amsthm}
\usepackage{graphicx}
\usepackage{hyperref}
\usepackage{cleveref}
\usepackage{color}

\usepackage[round]{natbib}

\newtheorem{thm}{Theorem}[section]
\newtheorem{prop}[thm]{Proposition}
\newtheorem{lem}[thm]{Lemma}

\theoremstyle{remark}
\newtheorem{rmk}[thm]{Remark}
\theoremstyle{definition}
\newtheorem{constr}[thm]{Construction}

\newcommand{\lactobounce}{\Xi_{\mathrm{bounce}}}
\newcommand{\lactosteep}{\Xi_{\mathrm{steep}}}
\newcommand{\lactoposet}{\Xi_{\mathrm{poset}}}
\newcommand{\bouncetolac}{\Lambda_{\mathrm{bounce}}}
\newcommand{\steeptolac}{\Lambda_{\mathrm{steep}}}
\newcommand{\posettolac}{\Lambda_{\mathrm{poset}}}
\newcommand{\pointdyck}{\varphi}
\newcommand{\latdyck}{\psi}

\newcommand{\trees}{\mathcal{T}}
\newcommand{\posets}{\mathcal{P}}
\newcommand{\dyck}{\mathcal{D}}

\newcommand{\cat}{\mathrm{Cat}}

\newcommand{\Area}{\operatorname{Area}}

\newcommand{\tdef}[1]{\emph{#1}}

\newcommand{\insertfig}[2]{\includegraphics[page=#1,width=#2\textwidth]{fig-ipe.pdf}}

\author{Wenjie Fang}
\affiliation{LIGM, Univ Gustave Eiffel, CNRS, ESIEE Paris, Marne-la-Vallée, France}
\keywords{Dyck path, zeta map, unit interval poset, plane tree\vspace*{-0.2cm}}
\title{Bijective proof of a conjecture on unit interval posets}
\begin{document}

\publicationdata{vol. 26:2}{2024}{2}{10.46298/dmtcs.10837}{2023-01-20; 2023-01-20; 2024-01-15}{2024-01-30}

\maketitle

\begin{abstract}
  In a recent preprint, Matherne, Morales and Selover conjectured that two different representations of unit interval posets are related by the famous zeta map in $q,t$-Catalan combinatorics. This conjecture was proved recently by Gélinas, Segovia and Thomas using induction. In this short note, we provide a bijective proof of the same conjecture with a reformulation of the zeta map using left-aligned colored trees, first proposed in the study of parabolic Tamari lattices.
\end{abstract}

\section{Introduction}

The study of unit interval posets started in statistics \citep{poset-counting} and psychology \citep{scott1964measurement}. However, it has since been connected to other objects in algebraic combinatorics, such as $(3+1)$-free posets \cite{anti-adjacency} and positroids \cite{pointdyck}. The connection to $(3+1)$-free posets is particularly important, as these posets are at the center of the Stanley--Stembridge conjecture in \citet{stanley-stembridge}, which states that the chromatic symmetric function of the incomparability graph of a $(3+1)$-free poset has only positive coefficients when expanded in the basis of elementary symmetric functions. Unit interval posets thus receive some attention as it can be used to study the structure of $(3+1)$-free posets \citep*{anti-adjacency,tangle,graded-3+1}.

It is known in \cite{poset-counting} that unit interval posets are counted by Catalan numbers, and researchers have given two different bijections to represent a unit interval poset by Dyck paths \citep*{anti-adjacency,tangle}. In a recent preprint \citep*{conj-zeta}, it was conjectured that the two bijections are related by the famous zeta map in $q,t$-Catalan combinatorics \citep{Haglund-qt-catalan}, which was first given in \citep*{zeta-orig}. In this short note, we settle this conjecture using bijections (see \Cref{thm:main-zeta}), in contrast to a recent inductive proof of the same conjecture by \citet*{alt-proof}. To this end, we first introduce a bijection $\lactoposet$ between unit interval posets and plane trees. We then show that the two different known bijections from unit interval posets are conjugates of special cases of some bijections $\lactosteep$ and $\lactobounce$ from \citet*{cataland} constructed for the study of parabolic Tamari lattices. Using the link between $\lactosteep, \lactobounce$ and the zeta map established in the same article, we conclude our proof.

The rest of the article is organized as follows. In \Cref{sec:prelim} we give the basic definitions. Then we give our bijection between unit interval posets and plane trees in \Cref{sec:lac}, along with the related bijections from \citet*{cataland}. Finally in \Cref{sec:dyck} we recall the two known bijections between unit interval posets and Dyck paths, and then establish their link with the bijections in \Cref{sec:lac}, leading to a bijective proof of our main result (\Cref{thm:main-zeta}).

\paragraph{Acknowledgment} We thank Adrien Segovia for bringing this conjecture to our attention. We also thank the anonymous reviewers for their detailed and helpful comments. This work is not supported by any funding with precise predefined goal, but it is supported by the publicly funded laboratory LIGM of Université Gustave Eiffel.

\section{Preliminary} \label{sec:prelim}

For convenience, we write $[n]$ for the set $\{1, 2, \ldots, n\}$. We consider finite partially order sets (or \tdef{poset} for short) of the form $P = (P_{\mathrm{elem}}, \preceq)$, where $P_{\mathrm{elem}}$ is a finite set and $\preceq$ a partial order on $P_{\mathrm{elem}}$. Given a set $S$ of real numbers $x_1 < x_2 < \cdots < x_n$, we define a partial order $\preceq_S$ on $[n]$ by taking $i \preceq_S j$ if and only if $x_i + 1 < x_j$. The order $\preceq_S$ can also be seen as defined on intervals of unit length starting at the $x_i$'s, where $i \preceq_S j$ if the interval starting with $x_i$ is on the left of that with $x_j$ without overlap. A poset $P$ with $n$ elements is a \tdef{unit interval order} if there is some $S \subset \mathbb{R}$ with $n$ elements such $P \cong ([n], \preceq_S)$. In this case, we call $S$ a \tdef{starting set} of $P$. We sometimes represent unit interval orders as $([n], \preceq_S)$ for some $S$ hereinafter.

We have the following characterization of unit interval orders.

\begin{thm}[{\citet[Theorem~2.1]{scott1964measurement}}] \label{unit-interval-order-chara}
  A poset $P$ is a unit interval poset if and only if it is $(3+1)$-free and $(2+2)$-free, that is, the order induced on any four elements cannot be a chain of $3$ elements plus an incomparable element, or two disjoint chains each containing $2$ elements.
\end{thm}

Unit interval orders are counted by Catalan numbers.

\begin{prop}[\citet{poset-counting}] \label{prop:poset-cat}
  The number of unit interval posets with $n$ elements is the $n$-th Catalan number $\cat_n = \frac1{2n+1} \binom{2n+1}{n}$.
\end{prop}

By definition, we may represent a unit interval poset $P$ by a set of real numbers $S$, though the choice of $S$ is clearly not unique. In the following, for convenience, we use this perspective of (starting points of) intervals, which is arguably easier to manipulate. We denote by $\posets_n$ the set of unit interval posets with $n$ elements.

There are many other families of combinatorial objects counted by Catalan numbers, such as Dyck paths and plane trees. A \tdef{Dyck path} is a lattice path formed by north steps $\uparrow \; = (0, 1)$ and east steps $\rightarrow \; = (1, 0)$, starting at $(0, 0)$ and ending on the diagonal $y = x$ while staying weakly above it. The \tdef{size} of a Dyck path is the number of its north steps. We denote by $\dyck_n$ the set of Dyck paths of size $n$. We define plane trees recursively: a \tdef{plane tree} $T$ is either a single node (a \tdef{leaf}) or an \tdef{internal node} $u$ linked by edges to a sequence of plane trees called \tdef{sub-trees}. In the latter case, the node $u$ is also called the \tdef{root} of $T$, and the roots of sub-trees are called \tdef{children} of $u$. We denote by $\trees_n$ the set of plane trees with $n$ non-root nodes. We recall the well-known fact that $|\trees_n| = |\dyck_n| = \cat_n$. For a node $u$ in a plane tree $T$, its \tdef{depth}, denoted by $d(u)$, is the distance (number of edges) between $u$ and the root of $T$.

\section{Unit interval posets and plane trees} \label{sec:lac}

We start by a new bijection between unit interval posets and plane trees.

\begin{constr} \label{constr:posettolac}
  Let $S = \{x_1 < \cdots < x_n\}$ be a starting set of a unit interval poset $P = ([n], \preceq_S)$. We define a plane tree $T$ as follows (see \Cref{fig:posettolac} for an example). We set $x_0 = x_1 - 2$. We denote by $v_i$ the node of $T$ corresponding to $x_i$. For $i \in [n]$, the parent of $v_i$ is $v_j$ if and only if $j$ is the largest index such that $j \preceq_S i$. By the definition of $x_0$, the parent of each $v_i$ is well-defined, and all nodes are descendants of $x_0$. We then order children of the same node from left to right with \emph{decreasing} index. We thus conclude that $T$ is a well-defined plane tree. We note that $T$ depends only on $\preceq_S$. We define $\posettolac(P) = T$.
\end{constr}

\begin{figure}
  \centering
  \insertfig{2}{1}
  \caption{Example of $\posettolac$ from a unit interval poset defined by a set of unit intervals to a plane tree}
  \label{fig:posettolac}
\end{figure}

\begin{constr} \label{constr:lactoposet}
  Given a plane tree $T$ with $n$ non-root nodes, we define an order relation $\preceq_S$ on $[n]$ induced by some $S \subset \mathbb{R}$. Let $r_T$ be the root of $T$, and $m$ the maximal arity of $T$, \textit{i.e.}, the maximal number of children of nodes in $T$. Given a node $u$ in $T$, if it is the $i$-th child of its parent \emph{from right to left}, then we define $c(u) = i$. We observe that $1 \leq c(u) \leq m$. For a non-root node $u$ in $T$, let $u_0 = r_T, u_1, \ldots, u_{d(u)} = u$ be the nodes on the unique path from the root $u_0 = r_T$ to $u_{d(u)} = u$, where $d(u)$ is the depth of $u$. We then define a real number $x_u$ associated to $u$ using base $(m+2)$:
  \begin{equation}
    \label{eq:node-number}
    x_u = d(u) + (0.c(u_1)c(u_2) \cdots c(u_{d(u)}))_{(m+2)} = \ell + \sum_{i=1}^{d(u)} c(u_i) (m+2)^{-i}.
  \end{equation}
  It is clear from \eqref{eq:node-number} that $x_u$ is strictly decreasing from left to right for nodes of the same depth in $T$. Let $S$ be the set of $x_u$ for non-root nodes $u$ in $T$. We define $\lactoposet(T) = ([n], \preceq_S)$.
\end{constr}

As an example of \Cref{constr:lactoposet}, let $T$ be the tree on the right part of \Cref{fig:posettolac}, with nodes labeled as in the figure, and $S$ the set of numbers in the construction. The maximal arity of $T$ is $m = 4$, achieved by the root. The number $x_{v_{12}} \in S$ for the node $v_{12}$ is thus $x_{v_{12}} = 3 + (0.321)_6 = 769/216$, as $v_{12}$ is of depth $3$, and it is the first child \emph{from right to left} of its parent $v_8$, itself the second child \emph{from right to left} of its parent $v_3$, which is the third child \emph{from right to left} of the root.

The following lemma gives a direct connection between a plane tree and its corresponding unit interval poset obtained from $\lactoposet$.

\begin{lem} \label{lem:parent-cond}
  For $T \in \trees_n$, we take $S$ to be the set constructed in \Cref{constr:lactoposet}, and $x_u$ the number constructed from a non-root node $u$ of $T$. Then for non-root nodes $u, v$, the parent of $u$ in $T$ is $v$ if and only if $x_v$ is the largest element in $S$ smaller than $1 + x_u$.
\end{lem}
\begin{proof}
  Suppose that $T$ has maximal arity $m$. Let $v'$ be the parent of $u$. By \eqref{eq:node-number}, we have $1 < x_u - x_{v'} < 1 + (m+2)^{-\ell+1}$, with $\ell$ the distance from $v'$ to the root. Then, $x_v$ satisfying $x_{v'} \leq x_v < 1 + x_u$ means $0 \leq x_v - x_{v'} < (m+2)^{-\ell+1}$, which means $v = v'$ by \eqref{eq:node-number}, as it requires $v$ and $v'$ to have the same distance to the root, but in this case $|x_v - x_{v'}|$ cannot be smaller than $(m+2)^{-\ell+1}$ if $v \neq v'$. We thus have the equivalence.
\end{proof}

We now show that $\posettolac$ and $\lactoposet$ are bijections and are inverse to each other.

\begin{prop} \label{prop:lac-poset-bij}
  For all $n \geq 1$, the map $\posettolac$ is a bijection from $\posets_n$ to $\trees_n$, with $\lactoposet$ its inverse.
\end{prop}
\begin{proof}
  By \Cref{prop:poset-cat}, we have $|\posets_n| = \cat_n = |\trees_n|$. We thus only need to show that $\posettolac \circ \lactoposet = \operatorname{id}$. Given a plane tree $T \in \trees_n$, let $P = ([n], \preceq_S) = \lactoposet(T)$ and $T' = \posettolac(P)$. Let $S = \{ x_1 < \cdots < x_n \}$ be the set of real numbers constructed in \Cref{constr:lactoposet} for $\preceq_S$. Given $i \in [n]$, let $u_i$ be the node in $T$ that gives rise to $x_i$ in $S$, and $u'_i$ the node in $T'$ that represents $x_i$. By \Cref{lem:parent-cond} and \Cref{constr:posettolac}, for any $i, j \in [n]$, the node $u_i$ is the parent of $u_j$ in $T$ if and only if $u'_i$ is the parent of $u'_j$ in $T'$. Then, $T'$ has the same order of siblings as $T$, as in both we order siblings with decreasing order in their corresponding real numbers in $S$. We thus conclude that $T = T'$.
\end{proof}

We now restate two previously known bijections between plane trees and Dyck paths that will be used to prove our main result.

\begin{constr} \label{constr:lactosteep}
  Let $n \geq 1$ and $T \in \trees_n$, we construct a Dyck path $D$ by taking the \emph{clockwise} contour walk starting from the top of the root of $T$ (see the Dyck path on the right of \Cref{fig:zeta-lac}). During the walk, when we pass an edge for the first (resp. second) time, we append $\uparrow$ (resp. $\rightarrow$) to $D$. We define $\lactosteep(T) = D$. The map $\lactosteep$ is a bijection from $\trees_n$ to $\dyck_n$ for all $n \geq 1$, and we denote its inverse by $\steeptolac$.
\end{constr}

\begin{constr} \label{constr:lactobounce}
  Let $n \geq 1$ and $T \in \trees_n$, we construct a Dyck path $D$ by appointing the number of north steps at each $x$-coordinate (see the Dyck path on the left of \Cref{fig:zeta-lac}). Let $\alpha_\ell$ be the number of nodes of depth $\ell$ (thus distance $\ell$ to the root), and $\ell_{\max}$ the maximal depth of nodes in $T$. We take $s_\ell = \alpha_1 + \cdots + \alpha_\ell$. Given $1 \leq k \leq n - 1$, there is a unique way to write $k = s_\ell - r$ with $0 \leq r < \alpha_{\ell}$ and $1 \leq \ell \leq \ell_{\max}$. In this case, the number of north steps in $D$ on $x = k$ is the number of children of the $(r+1)$-st node of depth $\ell$. The number of north steps in $D$ on $x = 0$ is the number of children of the root. We define $\lactobounce(T) = D$. The map $\lactobounce$ is a bijection from $\trees_n$ to $\dyck_n$ for all $n \geq 1$, see \citet*[Lemma~3.18]{cataland}, and we denote its inverse by $\bouncetolac = \lactobounce^{-1}$.
\end{constr}

\begin{prop}[{\citet*[Theorem~IV]{cataland}}] \label{prop:zeta-map}
  Let $\zeta$ be the zeta map from $\dyck_n$ to $\dyck_n$ with $n \geq 1$. We have $\zeta = \lactobounce \circ \steeptolac$.
\end{prop}

See \Cref{fig:zeta-lac} for an illustration of \Cref{prop:zeta-map}. Here, we do not give the original definition of the zeta map $\zeta$ in \citet*{zeta-orig}, and will instead be using \Cref{prop:zeta-map} in the following, as it is better suited to our approach.

\begin{figure}
  \centering
  \insertfig{3}{0.85}
  \caption{The zeta map as composition of bijections mediated by trees. The nodes of the same depth of the tree are grouped together. For the Dyck path on the left, the number of north steps on $x = k$ is the number of children of the $k$-th node in the tree, ordered by increasing depth, then from right to left. The one on the right comes from a clockwise contour walk.}
  \label{fig:zeta-lac}
\end{figure}

\begin{rmk}
  Albeit the same notation, the maps $\lactosteep$ and $\lactobounce$ are only special cases of the ones from \citet*[Construction~3.10~and~Equation~14]{cataland}. More precisely, the central objects of \citet*{cataland} are called \emph{LAC trees}, which are plane trees with non-root nodes divided into regions algorithmically. We are only using the special case where the region $d$ consists of nodes of depth $d$, which is also the one used in \citet*[Section~3.3]{cataland} to prove \Cref{prop:zeta-map}. In full generality, the map $\lactosteep$ maps a LAC tree to a \emph{steep pair}, which is a pair of nested Dyck paths with a \emph{steep} upper path, \textit{i.e.},without two consecutive east steps except on the maximal height. The map $\lactobounce$ sends a LAC tree to a \emph{bounce pair}, which is a pair of nested Dyck path with a \emph{bounce} lower path, \textit{i.e.}, a concatenation of sub-paths of the form $\uparrow^k \rightarrow^k$ for all $k$. This explains the words ``steep'' and ``bounce'' in their notation.
\end{rmk}

\begin{rmk}
  The bijection $\lactosteep$ is classical, except that we do the contour walk from right to left instead of from left to right. The bijection $\lactobounce$ is quite close to the classical bijection between plane trees and Łukasiewicz words, as given in \citet[Section~I.5.3]{flajolet}), which can be seen as sequences $(z_0, z_1, \ldots, z_{n-1})$ with $z_i \geq -1$ such that $\sum_{i=0}^k z_i \geq 0$ for $0 \leq k < n - 1$, and the sum of all $z_i$'s is $-1$. However, in the classical bijection of Łukasiewicz, we deal with nodes in a depth-first order, but in $\lactobounce$ they are processed in a breadth-first order.
\end{rmk}

\section{Unit interval posets and the zeta map} \label{sec:dyck}

In the following, we give the two representations of unit interval posets as Dyck paths, and we detail how they are related to the bijections detailed in \Cref{sec:lac}, leading to a proof of our main result (\Cref{thm:main-zeta}).

We start by the first representation, which was first defined implicitly using anti-adjacency matrices of unit interval posets in \cite{anti-adjacency}, only involving the poset structure. The following form was first given in \cite{pointdyck}.

\begin{constr} \label{constr:pointdyck}
  Given $P = ([n], \preceq_S)$ a unit interval poset with $S \subset \mathbb{R}$, we define a Dyck path $D$ as follows. Let $S^+ = \{x + 1 \mid x \in S\}$. Without loss of generality, we can choose $S$ such that $S \cap S^+ = \varnothing$. Then suppose that $S \cup S^+ = \{ y_1 < \cdots < y_{2n} \}$. For $i \in [2n]$, the $i$-th step of $D$ is $\uparrow$ if $y_i \in S$, and is $\rightarrow$ otherwise. We define $\pointdyck(P) = D$. See \Cref{fig:pointdyck-latdyck} for an example.
\end{constr}

This form of the map $\pointdyck$ in \Cref{constr:pointdyck} was given in \citet[Section~5]{pointdyck} and proved in Lemma~5.7 therein to be equivalent to the original definition, which is the one used in \citet*[Section~2.3]{conj-zeta}. Although defined here using the set $S$, the map $\pointdyck$ does not depend on the choice of $S$.

\begin{figure}
  \centering
  \insertfig{4}{0.8}
  \caption{Example of the bijections $\pointdyck$ and $\latdyck$}
  \label{fig:pointdyck-latdyck}
\end{figure}

\begin{prop} \label{prop:pointdyck-bounce}
  We have $\pointdyck \circ \lactoposet = \lactobounce$, and $\pointdyck$ is a bijection from $\posets_n$ to $\dyck_n$ for all $n \geq 1$.
\end{prop}
\begin{proof}
  Let $T \in \trees_n$ and $D = \lactobounce(T)$. We take $P = \lactoposet(T) = ([n], \preceq_S) \in \posets_n$ with $S = \{x_1 < \cdots < x_n\}$ given in \Cref{constr:lactoposet}. Let $D' = \pointdyck(P)$. To show that $D' = D$, we first notice that a Dyck path is determined by the number of north steps on each line $x = k$ for $0 \leq k \leq n - 1$. Thus, we only need to show that for each $k$, $D'$ has the same number of north steps on $x=k$ as $D$ has according to \Cref{constr:lactobounce}.

  For the line $x=0$, we observe that the number of elements in $S$ strictly smaller than $x_1 + 1$ must be those with integer part equal to $1$, thus those from the children of the root of $T$. For $1 \leq k \leq n-1$, from \Cref{constr:pointdyck}, we know that the number of north steps on the line $x=k$ is the number of elements $x_i$ such that $x_{k-1} + 1 < x_i < x_k + 1$, as $x_k+1$ corresponds to the $k$-th east step in $D$. Let $u_k$ be the node in $T$ corresponding to $x_k$. By \Cref{lem:parent-cond}, the nodes in $T$ corresponding to such $x_i$ are children of $u_k$, meaning that the total number of north steps on $x=k$ is the number of children of $u_k$. Furthermore, suppose that $u_k$ is of depth $\ell$, then we can write $k = s_\ell - r$ with $0 \leq r < \alpha_\ell$ as in \Cref{constr:lactobounce}. We know that $u_k$ is the $(r+1)$-st node of depth $\ell$ in $T$, as in \Cref{constr:posettolac} the values corresponding to nodes of the same depth are decreasing from left to right. We thus conclude that $D = D'$, and $\pointdyck \circ \lactoposet = \lactobounce$. As both $\lactobounce$ \citep*[Lemma~3.18]{cataland} and $\lactoposet$ (\Cref{prop:lac-poset-bij}) are bijections, we conclude that $\pointdyck$ is also a bijection.
\end{proof}

The second presentation was first defined in \citet*{tangle} for $(3+1)$-free posets, and it takes a simpler form for unit interval posets. A more explicit presentation is given in \citet[Section~2]{tangle-simple}.

\begin{constr} \label{constr:latdyck}
  For $n \geq 1$, let $D \in \dyck_n$ be a Dyck path of size $n$. Its \tdef{area vector}, denoted by $\Area(D)$, is a vector $(a_1, \ldots, a_n)$ with $a_i$ the number of full unit squares with the upper edge on $y = i$ between $D$ and the diagonal $y = x$. Such area vectors are characterized by the conditions $a_1=0$ and $a_i \leq a_{i-1} + 1$ for $2 \leq i \leq n$. It is clear that the area vector determines the Dyck path. We then define a poset $P = ([n], \preceq)$ by taking $i \prec j$ for $i, j \in [n]$ such that
  \begin{itemize}
  \item Either $a_i + 2 \leq a_j$;
  \item Or $a_i + 1 = a_j$ and $i < j$.
  \end{itemize}
  We define $\latdyck(D) = P$, and its inverse $\latdyck^{-1}(P) = D$. See \Cref{fig:pointdyck-latdyck} for an example.
\end{constr}

The following result ensures that $\latdyck$ is well-defined. It is a special case of \citet*[Theorem~5.2]{conj-zeta}, which is obtained by combining \citet*[Remark~3.2, Proposition~3.11]{tangle}. A self-contained direct proof of this special case can be found in \citet*[Section~6]{alt-proof}.

\begin{prop}[Special case of {\citet*[Theorem~5.2]{conj-zeta}}] \label{prop:latdyck}
  The map $\latdyck$ is a bijection from Dyck paths to unit interval posets preserving sizes.
\end{prop}

\begin{prop}\label{prop:latdyck-steep}
  We have $\latdyck = \lactoposet \circ \steeptolac$.
\end{prop}
\begin{proof}
  Let $P = \latdyck(D)$, $T = \steeptolac(D)$ and $P' = \lactoposet(D)$. We write $P = ([n], \preceq)$ and $P' = ([n], \preceq_S)$, with $S$ given in \Cref{constr:lactoposet}. We show that $P$ is isomorphic to $P'$ after an implicit relabeling.

  For $i \in [n]$, let $u_i$ be the $i$-th non-root node in $T$ that is visited in the clockwise contour walk of $T$. We note that this order on nodes of $T$ is different from the one used before, for instance in \Cref{constr:posettolac}. Suppose that $\Area(D) = (a_1, \ldots, a_n)$. By \Cref{constr:lactosteep}, the depth of $u_i$ is $a_i + 1$. We now define a partial order $\preceq_T$ on non-root nodes in $T$ by taking $u_i \preceq_T u_j$ if and only if $i \preceq j$ in $P$. We recall that the depth of $v$ in $T$ is denoted by $d(v)$. By \Cref{constr:latdyck}, two non-root nodes $v, w$ in $T$ satisfy $v \prec_T w$ if and only if
  \begin{itemize}
  \item Either $d(v) + 2 \leq d(w)$;
  \item Or $d(v) + 1 = d(w)$, and the parent of $w$ is either $v$ or on the left of $v$.
  \end{itemize}
  Now, for $v$ a non-root node in $T$, we take $x_v$ defined in \eqref{eq:node-number}. Suppose that there are two non-root nodes $v, w$ of $T$ such that $x_v + 1 < x_w$. Then we have $d(v) + 1 \leq d(w)$. There are two possibilities: either $d(v) + 1 = d(w)$, or $d(v) + 2 \leq d(w)$. In the first case, let $w'$ be the parent of $w$, we have $x_{w'} + 1 < x_w$. By \Cref{lem:parent-cond}, we know that $x_v \leq x_{w'}$. From \eqref{eq:node-number}, we know that either $v = w'$ or $v$ is on the right of $w'$, as $x_v$ is decreasing from left to right for nodes in $T$ of the same depth. Combining with the second case, we conclude that $x_v + 1 < x_w$ if and only if $v \prec_T w$. By \Cref{constr:lactoposet}, we have $P \cong P'$, which concludes the proof.
\end{proof}

\begin{thm}[{\citet*[Conjecture~7.1]{conj-zeta}}] \label{thm:main-zeta}
  We have $\pointdyck \circ \latdyck = \zeta$.
\end{thm}
\begin{proof}
  Combining \Cref{prop:pointdyck-bounce}~and~\ref{prop:latdyck-steep}, we have $\pointdyck \circ \latdyck = \lactobounce \circ \steeptolac$, and we conclude by \Cref{prop:zeta-map}.
\end{proof}

\bibliographystyle{abbrvnat}
\bibliography{unit-interval-order}

\begin{thebibliography}{14}
\providecommand{\natexlab}[1]{#1}
\providecommand{\url}[1]{\texttt{#1}}
\expandafter\ifx\csname urlstyle\endcsname\relax
  \providecommand{\doi}[1]{doi: #1}\else
  \providecommand{\doi}{doi: \begingroup \urlstyle{rm}\Url}\fi

\bibitem[Andrews et~al.(2002)Andrews, Krattenthaler, Orsina, and
  Papi]{zeta-orig}
G.~E. Andrews, C.~Krattenthaler, L.~Orsina, and P.~Papi.
\newblock ad-nilpotent {$\mathfrak b$}-ideals in {${\rm sl}(n)$} having a fixed
  class of nilpotence: {C}ombinatorics and enumeration.
\newblock \emph{Trans. Amer. Math. Soc.}, 354\penalty0 (10):\penalty0
  3835--3853, 2002.

\bibitem[Ceballos et~al.(2020)Ceballos, Fang, and M{\"u}hle]{cataland}
C.~Ceballos, W.~Fang, and H.~M{\"u}hle.
\newblock {The Steep-Bounce Zeta Map in Parabolic Cataland}.
\newblock \emph{J. Combin. Theory Ser. A}, 172:\penalty0 105210, 2020.
\newblock \doi{10.1016/j.jcta.2020.105210}.

\bibitem[Chavez and Gotti(2018)]{pointdyck}
A.~Chavez and F.~Gotti.
\newblock Dyck paths and positroids from unit interval orders.
\newblock \emph{J. Combin. Theory Ser. A}, 154:\penalty0 507--532, 2018.
\newblock \doi{10.1016/j.jcta.2017.09.005}.

\bibitem[Flajolet and Sedgewick(2009)]{flajolet}
P.~Flajolet and R.~Sedgewick.
\newblock \emph{Analytic combinatorics}.
\newblock Cambridge University Press, Cambridge, 2009.
\newblock ISBN 978-0-521-89806-5.
\newblock \doi{10.1017/CBO9780511801655}.

\bibitem[Guay-Paquet(2013)]{tangle-simple}
M.~Guay-Paquet.
\newblock A modular relation for the chromatic symmetric functions of
  $(3+1)$-free posets.
\newblock arXiv:1306.2400 [math.CO], 2013.

\bibitem[Guay-Paquet et~al.(2014)Guay-Paquet, Morales, and Rowland]{tangle}
M.~Guay-Paquet, A.~H. Morales, and E.~Rowland.
\newblock {Structure and Enumeration of $(3+1)$-Free Posets}.
\newblock \emph{Ann. Comb.}, 18\penalty0 (4):\penalty0 645--674, 2014.
\newblock \doi{10.1007/s00026-014-0249-2}.

\bibitem[Gélinas et~al.()Gélinas, Segovia, and Thomas]{alt-proof}
F.~Gélinas, A.~Segovia, and H.~Thomas.
\newblock Proof of a conjecture of {M}atherne, {M}orales, and {S}elover on
  encodings of unit interval orders.
\newblock arXiv:2212.12171 [math.CO].

\bibitem[Haglund(2008)]{Haglund-qt-catalan}
J.~Haglund.
\newblock \emph{{The $q,t$-Catalan Numbers and the Space of Diagonal
  Harmonics}}, volume~41.
\newblock American Mathematical Society, Providence, RI, 2008.

\bibitem[Lewis and Zhang(2013)]{graded-3+1}
J.~B. Lewis and Y.~X. Zhang.
\newblock Enumeration of graded {$(\bold{3}+\bold{1})$}-avoiding posets.
\newblock \emph{J. Combin. Theory Ser. A}, 120\penalty0 (6):\penalty0
  1305--1327, 2013.
\newblock \doi{10.1016/j.jcta.2013.03.012}.

\bibitem[Matherne et~al.()Matherne, Morales, and Selover]{conj-zeta}
J.~P. Matherne, A.~H. Morales, and J.~Selover.
\newblock {The Newton polytope and Lorentzian property of chromatic symmetric
  functions}.
\newblock arXiv:2201.07333 [math.CO].

\bibitem[Scott(1964)]{scott1964measurement}
D.~Scott.
\newblock Measurement structures and linear inequalities.
\newblock \emph{J. Math. Psychol.}, 1\penalty0 (2):\penalty0 233--247, 1964.

\bibitem[Skandera and Reed(2003)]{anti-adjacency}
M.~Skandera and B.~Reed.
\newblock Total nonnegativity and $(3+1)$-free posets.
\newblock \emph{J. Combin. Theory Ser. A}, 103\penalty0 (2):\penalty0 237--256,
  2003.
\newblock \doi{10.1016/s0097-3165(03)00072-4}.

\bibitem[Stanley and Stembridge(1993)]{stanley-stembridge}
R.~P. Stanley and J.~R. Stembridge.
\newblock On immanants of {Jacobi-Trudi} matrices and permutations with
  restricted position.
\newblock \emph{J. Combin. Theory Ser. A}, 62\penalty0 (2):\penalty0 261--279,
  1993.
\newblock \doi{10.1016/0097-3165(93)90048-d}.

\bibitem[Wine and Freund(1957)]{poset-counting}
R.~L. Wine and J.~E. Freund.
\newblock {On the Enumeration of Decision Patterns Involving $n$ Means}.
\newblock \emph{Ann. Math. Stat.}, 28\penalty0 (1):\penalty0 256--259, 1957.
\newblock \doi{10.1214/aoms/1177707051}.

\end{thebibliography}

\end{document}